\documentclass[final,leqno]{siamltex} 
\usepackage{amssymb}
\usepackage[reqno]{amsmath}
\usepackage{charter}
\usepackage[a4paper]{geometry}

\usepackage{esint}
\usepackage{empheq}

\newcommand{\bld}[1]{{\bf{#1}}}
\newcommand{\norm}[1]{\left\Vert#1\right\Vert}

\def\R{\mathbb{R}}
\def\E{\mathbb{E}}

\begin{document}

\title{Three improvements to multi-level Monte Carlo simulation of SDE systems\thanks{This work was performed under the auspices of the U.S. DOE by the University of California, Los Angeles, under grant DE-FG02-05ER25710.}}
\author{L.F. Ricketson\thanks{Mathematics Department, University of California at Los Angeles, lfr@math.ucla.edu}}
\date{\today}

\maketitle

\begin{abstract}
We introduce three related but distinct improvements to multilevel Monte Carlo (MLMC) methods for the solution of systems of stochastic differential equations (SDEs).  Firstly, we show that when the payoff function is twice continuously differentiable, the computational cost of the scheme can be dramatically reduced using a technique we call `Ito linearization'.  Secondly, by again using Ito's lemma, we introduce an alternative to the antithetic method of Giles et.\ al [M.B. Giles, L. Szpruch. arXiv preprint arXiv:1202.6283, 2012] that uses an approximate version of the Milstein discretization requiring no L\'{e}vy area simulation to obtain the theoretically optimal cost-to-error scaling.  Thirdly, we generalize the antithetic method of Giles to arbitrary refinement factors.  We present numerical results and compare the relative strengths of various MLMC-type methods, including each of those presented here.
\end{abstract}

\begin{keywords}
stochastic differential equations, multi-level, Monte Carlo, Ito's lemma
\end{keywords}

\begin{AMS}
65C05, 65C30
\end{AMS}

\section{Introduction}
Stochastic differential equations (SDEs) have numerous applications: neuroscience \cite{burkitt2006review,kallianpur1987weak}, chemical kinetics \cite{gillespie2000chemical}, civil engineering  \cite{haghighat1988predictive, harris1977modelling},  biological fluid dynamics \cite{kramer2008foundations}, physics \cite{lemons2009small, milstein2004stochastic}, and finance \cite{shreve2004stochastic}, to name a few.  A prototypical class of problems may be characterized as follows: let $\bld{S}(t) \in \R^d$ satisfy the system of SDEs 
\begin{equation} \label{SDE}
	dS_i = a_i(\bld{S},t) \, dt + \sum_{j=1}^D b_{ij}(\bld{S},t) \, dW_j, \qquad \bld{S}(0) = \bld{S}_0
\end{equation}
for $t \in [0, T]$ and some given $\bld{S}_0$, where $S_i$ is the $i^\textrm{th}$ component of $\bld{S}$, $W(t) \in \R^D$ is a $D$ dimensional Brownian motion, $a_i: \R^d \rightarrow \R$ for each $i \in \{1, 2, ..., d\}$, and similarly for $b_{ij}$.  Then, for some given $P: \R^d \rightarrow \R$, evaluate $\E [P(\bld{S}(T))]$.  That is, we wish to find the mean value of some functional of the solution of an SDE.  

Since exact solutions are available for only the simplest of SDEs, finite difference methods are frequently used to approximate their solutions.  The expectation is then evaluated via a Monte Carlo method.  The purpose of the present work is to present three improvements to the class of multilevel Monte Carlo (MLMC) methods - introduced in \cite{giles2008multilevel} - which are the current state of the art.  

The MLMC methods themselves improve upon the most straightforward numerical method for the archetypal SDE problem above.  That method is to approximate the SDE's solution by the well-known Euler-Maruyama discretization with time step $h$, given by
\begin{equation} \label{eulerdisc}
	S_{i,n+1} = S_{i,n} + a_i(\bld{S}_n, t_n) h + \sum_{j=1}^D b_{ij}(\bld{S}_n, t_n) \Delta W_{j,n},
\end{equation}
where $\bld{S}_n$ approximates $\bld{S}(t_n)$, with $t_n = nh$, and the $\Delta W_{j,n}$ are independent normal random variables with mean zero and variance $h$.  We may then generate $N$ independent samples of $\bld{S}_{T/h}$ by generating different $\Delta W_{j,n}$ for each sample, and estimate the desired expectation by
\begin{equation}
	\E[P(\bld{S}(T))] \approx \frac{1}{N} \sum_{r=1}^N P\left(\bld{S}_{T/h}^{(r)}\right),
\end{equation}
where $r$ indexes the $N$ samples.  

One desires to approximate the true expectation to within an RMS error $\varepsilon$, which will scale as $O(N^{-1/2})$ and $O(h)$.  The computational cost of the scheme is proportional to the total number of time steps taken, which scales as $O(N/h)$.  Thus, we see that the computational cost of achieving an RMS error $\varepsilon$ - which we henceforth denote by $K$ - is $O(\varepsilon^{-3})$.  

In many contexts, such a scaling is prohibitive, so a number of methods which improve upon it have been developed.  To understand them, we must define the notions of strong and weak errors for SDE approximations.  Let $\bld{S}_h$ be an approximate solution of (\ref{SDE}) obtained by some discretization with time-step $h$.  We say that discretization has \textit{weak error} of order $p$ if
\begin{equation}
	|\E[g(\bld{S}_h)] - \E[g(\bld{S})] | = O(h^p)
\end{equation}
for some broad class of functions $g: \R^d \rightarrow \R$ (in particular, that class should include $P$).  We say that discretization has \textit{strong error} of order $q$ if
\begin{equation} \label{strongscale}
	\E\left[ |\bld{S}_h - \bld{S}| \right] = O(h^q).
\end{equation}
We note that the Euler discretization has $p=1$ and $q=1/2$ \cite{kloeden2011numerical}.  

It is straightforward to see that if we modify the naive scheme presented above to use a discretization of weak order $p$, we have
\begin{equation}
	K = O\left(\varepsilon^{-(2 + 1/p)}\right),
\end{equation}
independent of $q$.  In contrast, the multilevel Monte Carlo (MLMC) methods introduced in \cite{giles2008multilevel} and expanded in \cite{giles2008improved, giles2012antithetic} achieve
\begin{equation} \label{MLMCscalings}
   K = \left\{
     \begin{array}{lr}
       O\left(\varepsilon^{-2} (\log \varepsilon)^2\right) & : q = 1/2 \\
       O\left( \varepsilon^{-2} \right) & : q>1/2
     \end{array}
   \right.
\end{equation}
so long as $p>0$.  The proof of this fact may be found in \cite{giles2008multilevel}, and we sketch the argument in section 2.  

We thus see that the multilevel method scales better than the naive method outlined above for \textit{any} discretization with finite weak order $p$.  Moreover, the larger the weak order of a discretization, the more regularity we require of $P$ to achieve that order \cite{kloeden2011numerical}, further limiting the use of high-order weak schemes.  Multilevel schemes are thus a great improvement over simple schemes of the type outlined above.

The MLMC schemes achieve their improved cost scaling by approximating the SDE's solution with many different time-steps (called `levels') and taking advantage of the discretization's strong convergence to get low variance estimates of the difference in the payoff at adjacent levels.  The remaining high variance quantity - the payoff's expectation at the lowest level - is relatively cheap to compute because of the large time-step.  However, the algorithm could be further improved by also applying a variance reduction at this lowest level.  The first contribution of the present work is to show that, when the payoff function is twice continuously differentiable, we can reduce the variance at the lowest level to zero by finding the payoff using Ito's lemma instead of direct evaluation.  

Our second contribution is to again make use of Ito's lemma to derive a variant of the MLMC method that achieves the cost scaling $O(\varepsilon^{-2})$ in spite of having $q=1/2$.  This is a desirable result because discretizations with $q>1/2$ require the simulation of L\'{e}vy areas when $D>1$, and L\'{e}vy areas are notoriously difficult to sample.  Indeed, no suitable algorithm has been implemented for $D>2$.  A method achieving $O(\varepsilon^{-2})$ scaling without L\'{e}vy area simulation was also derived in \cite{giles2012antithetic}.  However, the method we propose, while similar in some respects, is simpler to derive and slightly faster for twice differentiable payoffs.  

Thirdly, we make use of our analysis to generalize the antithetic method in \cite{giles2012antithetic} to arbitrary refinement factor - that is, the ratio between the time-steps at adjacent levels.  The method was originally derived for the case of refinement factor $M = 2$, but we show that $M \approx 4$ to $5$ is optimal.  Importantly, the generalization to arbitrary $M$ still requires the sampling of only one antithetic path, so the generalization introduces no extra computational complexity.  The key lemma in this development - Lemma \ref{antithet} in the present work - was originally proved in \cite{giles2012antithetic2} toward a different end.  Given this lemma, the result is straightforward, but does not appear elsewhere in the literature to the author's knowledge.  

The remainder of the paper is structured as follows.  Section 2 reviews the details of MLMC methods and the difficulty in implementing SDE solvers with $q>1/2$, focusing in particular on the Milstein discretization.  In section 3, we show how Ito's lemma can be used to eliminate the lowest level variance in MLMC methods.  In section 4, we use the results of the previous section to derive an `approximate Milstein' version of the MLMC method that achieves the $O(\varepsilon^{-2})$ cost scaling. In section 5, we leverage results from the previous section to generalize the antithetic method of \cite{giles2012antithetic}.  In section 6, we summarize results and present pseudocode for the algorithms proposed in previous sections.  In section 7, we present and discuss numerical results.  We conclude in section 8.  

\section{Background}
The first portion of this section reviews the derivation and basic properties of MLMC methods, while the second reviews the Milstein discretization, the difficulties inherent in its implementation, and some previous efforts to negotiate those difficulties.  For more details on elementary MLMC, see \cite{giles2008multilevel}.  For more information on Milstein, see \cite{gaines1994random,kloeden1992approximation,wiktorsson2001joint}.

\subsection{MLMC Review}
The MLMC schemes are constructed in the following way: for some integer $M > 1$, let $h_l = T M^{-l}$ for $l = 0, 1, 2, ..., L$. Setting $P_l = P(\bld{S}_{h_l}(T))$, the following identity holds:
\begin{equation} \label{fundMLMCid}
	\E \left[ P_L \right] = \E \left[ P_0 \right] + \sum_{l=1}^L \E \left[ P_l - P_{l-1} \right].  
\end{equation}
The weak convergence of the discretization guarantees that $\E [ P_L ]$ differs from the true expectation by $O(h_L^p)$, and (\ref{fundMLMCid}) shows that it can be estimated by estimating the $L+1$ expectations on the right side.  The first term is relatively cheap to compute, since the time-step $h_0 = T$ is much larger than $h_L$.  Meanwhile, the quantities $P_l - P_{l-1}$ have variances controlled by the strong convergence of the discretization, so that their expectations can be estimated accurately with a relatively small number of samples.   

We make this more concrete by defining
\begin{equation}
	V_l = \textrm{Var}\left[ P_l - P_{l-1} \right],
\end{equation}
for $l > 0$, where $\textrm{Var}[\cdot]$ denotes the variance of a random variable, and assuming $P$ has a global Lipschitz bound.  Then, if $P_l$ and $P_{l-1}$ are sampled using the same Brownian paths, we have
\begin{equation} \label{vartostrongarg}
\begin{split}
	V_l &= \E \left[ (P_l - P_{l-1})^2 \right] - \E \left[ P_l - P_{l-1} \right]^2 \\
	&\lesssim \E \left[ | \bld{S}_{h_l} - \bld{S}_{h_{l-1}} |^2 \right] + O\left(h_l^{2p}\right) \\
	&= O\left(h_l^{2q}\right) + O\left(h_l^{2p}\right).
\end{split}
\end{equation}
It is a general feature of SDE finite difference methods that $p \geq q$ \cite{kloeden2011numerical}, so we will write 
\begin{equation}
	V_l = O \left(h_l^{2q}\right)
\end{equation}
henceforth.  

If we estimate $\E \left[P_l - P_{l-1} \right]$ with $N_l$ samples - that is
\begin{equation}
	\E \left[P_l - P_{l-1} \right] \approx \hat{Y}_l \equiv \frac{1}{N_l} \sum_{r=1}^{N_l} \left(P_l^{(r)} - P_{l-1}^{(r)}\right),
\end{equation}
where $r$ again indexes the $N_l$ samples - then the variance in this estimate is $V_l/N_l$.  Similarly, define $V_0 = \textrm{Var} [P_0]$ and let 
\begin{equation}
	\hat{Y}_0 = \frac{1}{N_0} \sum_{r=1}^{N_0} P_{0}^{(r)}.
\end{equation}
Then, let $\hat{P}_L$ be our estimate of $\E[P_L]$ defined by
\begin{equation}
	\hat{P}_L = \sum_{l=0}^L \hat{Y}_l.
\end{equation}
This estimate has variance 
\begin{equation}
	\textrm{Var} \left[\hat{P}_L \right] = \sum_{l=0}^L \frac{V_l}{N_l}.
\end{equation}

The desired RMS error bound of $\varepsilon$ may thus be written as
\begin{equation} \label{MLMCerrors}
	(c_1 h_L)^2 + \sum_{l=0}^L \frac{V_l}{N_l} \leq \varepsilon^2,
\end{equation}
where $c_1$ is the constant of proportionality in the weak error estimate of the SDE scheme.  That is, 
\begin{equation}
	|\E[ P(\bld{S}(T)) - \E[P_L] | \approx c_1 h_L
\end{equation}
for sufficiently small $h_L$.  Note that we assume the scheme is first order in the weak sense ($p=1$), a quality shared by all the schemes considered in this paper.  We call the first term in (\ref{MLMCerrors}) the \textit{bias error}; it is deterministic and arises from the finite time-step approximation of the SDE's solution.  We call the second term - the sum - the \textit{sampling error}; it arises from the estimation of expectations using a finite number of samples.  

In the analysis of Giles, (\ref{MLMCerrors}) is satisfied by setting each of the two mean squared errors to $\varepsilon^2/2$.  The bias error constraint then immediately gives a formula for $L$, the total number of levels to be used:
\begin{equation}
	L = \left\lceil \frac{\log \left( \sqrt{2} c_1T/ \varepsilon \right)}{\log M} \right\rceil.
\end{equation}
The sampling error constraint gives rise to a constrained optimization problem: one wishes to minimize the computational cost - modeled by the total number of time steps taken - 
\begin{equation}
	K \propto \sum_{l=0}^L N_l(h_l^{-1} + h_{l-1}^{-1}) = \left( 1 + \frac{1}{M} \right) \sum_{l=0}^L \frac{N_l}{h_l},
\end{equation}
subject to the constraint $\sum_{l=0}^L (V_l / N_l) \leq \varepsilon^2/2$.  A Lagrange multiplier argument shows that the optimal choice is
\begin{equation} \label{optimalsamplenumbers}
	N_l = \frac{2}{\varepsilon^2} \sqrt{V_l h_l} \left( \sum_{l=0}^{L} \sqrt{V_l/h_l} \right),
\end{equation}
which in turn gives the cost 
\begin{equation} \label{costformula}
	K \propto \frac{2}{\varepsilon^2} \left( 1 + \frac{1}{M} \right) \left( \sum_{l=0}^{L} \sqrt{V_l/h_l} \right)^2.
\end{equation}

When $q=1/2$, we have $V_l = O(h_l)$, so that each term in the sum is $O(1)$, making the sum $O(L)$.  Since $L$ scales like $\log \varepsilon$, we see that $K = O(\varepsilon^{-2} (\log \varepsilon)^2)$, as stated in the introduction.  When $q>1/2$, the terms in the sum decrease geometrically, so that the sum to $L$ is bounded by a convergent infinite sum, giving $K = O(\varepsilon^{-2})$.  

In practice, the constant $c_1$ is not known, so $L$ cannot be specified at the start of the simulation.  One typically performs the necessary steps for $L=1$, estimates the bias error by looking at $\hat{Y}_L$, and increments $L$ while the bias error is estimated to be more that $\varepsilon/\sqrt{2}$.  More details can be found in \cite{giles2008multilevel} and in section 6 of this paper.  

\subsection{Milstein and L\'{e}vy Areas}
The simplest finite difference scheme for SDEs achieving $q>1/2$ - and thus yielding the optimal MLMC scaling - is the Milstein scheme, written as
\begin{equation} \label{milstein}
	S_{i,n+1} = S_{i,n} + a_{i,n} \Delta t + \sum_{j=1}^D b_{ij,n} \Delta W_{j,n} + \sum_{j,k=1}^D h_{ijk,n} (\Delta W_{j,n} \Delta W_{k,n} - \Omega_{jk} \Delta t - A_{jk,n}), 
\end{equation}
where we've abbreviated $a_i (\bld{S}_n, t_n) = a_{i,n}$ and similarly for $b_{ij,n}$ and $h_{ijk,n}$, $\Omega_{jk}$ is the correlation matrix associated with $W$, and $h$ and $A$ are defined by
\begin{equation} \label{hdef}
	h_{ijk} = \frac{1}{2} \sum_{l=1}^d b_{lk} \frac{\partial b_{ij}}{\partial x_l}, 
\end{equation}
\begin{equation}
	A_{jk,n} = \int_{t_n}^{t_{n+1}} \int_{t_n}^s \left[ dW_j(u) dW_k(s) - dW_k(u) dW_j(s) \right].
\end{equation}

The $A_{jk,n}$ are known as L\'{e}vy areas.  When $D=1$, they vanish, since $A_{jj,n} = 0$, and Milstein is straightforward to implement.  When $D=2$, there is effectively only one non-zero L\'{e}vy area, since $A_{jk,n} = -A_{kj,n}$.  Recently, an efficient method has been developed for sampling a single L\'{e}vy area \cite{dimits2013higher}, making Milstein implementation feasible when $D=2$.  Sampling multiple L\'{e}vy areas is a more challenging problem because they are not independent.  A method for jointly sampling multiple L\'{e}vy areas was also proposed in \cite{dimits2013higher, gaines1994random} that builds upon the methods therein and involves sampling a random orthogonal matrix, techniques for which are available in \cite{mezzadri2007generate}. However, this method has not been implemented or tested with MLMC.  As a result, implementing the Milstein discretization and thus achieving the $O(\varepsilon^{-2})$ scaling for MLMC methods is quite challenging when $D>2$, except in special cases.  

Fortunately, in \cite{giles2012antithetic} it was observed that while $q>1/2$ is sufficient to achieve the optimal scaling, it is not necessary.  The necessary condition is
\begin{equation}
	\textrm{Var}[P_l - P_{l-1}] = O\left(h_l^\beta \right)
\end{equation}
for some $\beta > 1$.  We can see from (\ref{vartostrongarg}) that if $P$ has a global Lipschitz bound, this necessary condition is achieved if
\begin{equation} \label{Sscale}
	\E \left[ |\bld{S}_{h_l} - \bld{S}_{h_{l-1}} |^2 \right] = O\left( h_l^{\beta} \right).
\end{equation}
This resembles a strong scaling requirement (\ref{strongscale}), but there is a key difference.  Here, we require two approximate solutions to be within $O(h_l^\beta)$ of each other in the mean square sense.  It is not necessary that \textit{either one} of these approximate solutions be within $O(h_l^\beta)$ of the \textit{true} solution, as would be the case if we were relying on strong convergence.  

In \cite{giles2012antithetic}, the Milstein scheme (\ref{milstein}) with the L\'{e}vy areas set to zero, along with an antithetic path sampling method, is used in order to achieve (\ref{Sscale}) with $\beta>1$, and thus achieve the $O(\varepsilon^{-2})$ cost scaling for SDE systems with arbitrary $D$.  For the moment, we refer the reader to that paper for its detailed derivation and implementation.  We will discuss some key aspects of the antithetic method as they become relevant in the course of our discussion here.  

In section 4 of this paper, we derive an alternative method to that in \cite{giles2012antithetic}.  We also achieve the $O(\varepsilon^{-2})$ cost scaling for arbitrary $D$ without simulating L\'{e}vy areas.  Our method requires more regularity of the payoff function, but is slightly cheaper and simpler to derive.  In section 5, we generalize the results of \cite{giles2012antithetic} to $M>2$.  Since much of the analysis from \cite{giles2012antithetic} carries over directly, we simply cite several results without reprinting proofs.  


\section{Variance Reduction via Ito's Lemma}
We begin with a simple observation.  Suppose that $P(\bld{S}) = S_m$, for some $1 \leq m \leq d$.  That is, $P$ simply picks out one of the components of $\bld{S}$.  Such a payoff function is useful in chemical kinetics, for example, in which each component of the SDE represents the concentration of a particular species and we may desire to compute the mean concentration of some key compound.  

Then, we may write a simple analytic expression for $P_0$ - the payoff when the time-step is $T$ - when the Euler discretization is used:
\begin{equation}
	P_0 = S_{m,0} + a_m (\bld{S}_0) T + \sum_{j=1}^D b_{mj}(\bld{S}_0) W_j(T),
\end{equation}
where $\bld{S}_0$ is the initial data and $S_{m,0}$ is its $m^{\textrm{th}}$ component.  The expectation of this expression is simple to evaluate:
\begin{equation} \label{P0eval}
	\E[P_0] = S_{m,0} + a_m (\bld{S}_0) T.
\end{equation}
The same result applies to the Milstein scheme, since the additional term has zero expectation.  Thus, when $P$ has this simple form (or, indeed, is any linear function of $\bld{S}$), the base payoff can be evaluated exactly in terms of the initial condition.  There is no need to sample any random variables at all.  In effect, $V_0 = 0$ and $N_0 = 1$.  

This can represent a great computational saving for MLMC because, as already noted, the lowest level is the only level at which no variance reduction is gained.  That is, with the standard approach, $V_0$ need not obey the same scaling as the other $V_l$, and may very well be disproportionately large, thus causing the cost of computing $P_0$ to dominate other costs.  

To illustrate this point, we show in fig.\ 1 the fraction of the computational work at each level in a sample MLMC computation, using both the Euler and antithetic methods.  We see that, for each method, the base level (zero) is the most expensive.  The base level represents an even larger fraction of the work in the antithetic method.  This is a result of the improved variance scaling, which reduces the cost of the higher levels.  

\begin{figure}
  \centering
	\includegraphics[width=.6\textwidth]{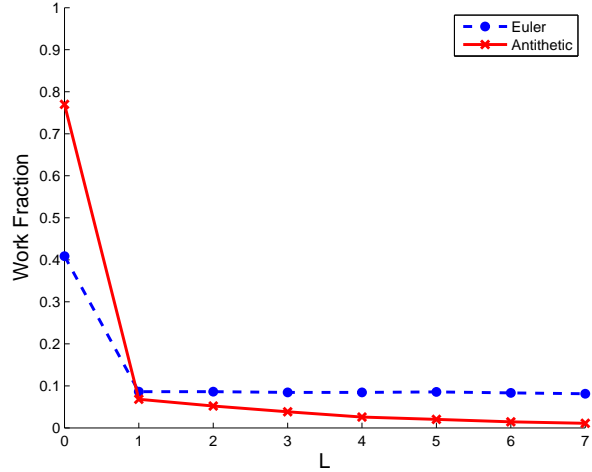}
	\caption{The fraction of the computational work exerted at each level in a sample MLMC computation.  The Heston model - see (\ref{heston}) and proceeding text for specification - is solved with a sinusoidal payoff function, and $M=2$.}
\end{figure} 

There is thus a motivation to investigate whether the technique of eliminating the cost of computing the base level payoff can be generalized to less trivial payoff functions.  
Toward that end, assume $P$ is twice continuously differentiable.  Then, Ito's lemma gives an SDE for $P$:
\begin{equation} \label{ito}
	dP = \left( \sum_{i=1}^d a_i P_{x_i} + \frac{1}{2} \sum_{j=1}^D \sum_{i,k=1}^d b_{ij} b_{kj} P_{x_i x_k} \right) \, dt + \sum_{j=1}^D \sum_{i=1}^d b_{ij} P_{x_i} \, dW_j, 
\end{equation}
where subscripts on $P$ denote partial derivatives, and all functions are evaluated at $\bld{S}(t)$. 

With this in mind, construct a vector $\mathcal{S} \in \R^{d+1}$ as follows: for $1 \leq k \leq d$, set $\mathcal{S}_k = S_k$, and set $\mathcal{S}_{d+1} = P(\bld{S})$.  Then, $\mathcal{S}$ solves
\begin{equation} \label{modSDE}
	d\mathcal{S}_i = \alpha_{i} \, dt + \sum_{j=1}^D \beta_{ij} \, dW_j
\end{equation}
where $\alpha_i(\mathcal{S}) = a_i(\bld{S})$ and $\beta_{ij}(\mathcal{S}) = b_{ij}(\bld{S})$ for $i \leq d$, and 
\begin{equation} \label{abdefs}
	\alpha_{d+1} (\mathcal{S}) = \sum_{i=1}^d a_i P_{x_i} + \frac{1}{2} \sum_{j=1}^D \sum_{i,k=1}^d b_{ij} b_{kj} P_{x_i x_k}, \qquad \beta_{(d+1)j} (\mathcal{S}) = \sum_{i=1}^d b_{ij} P_{x_i}.
\end{equation}
This is a system of SDEs in the usual sense.  Consider further the ``payoff" function $\tilde{P}(\mathcal{S}) = \mathcal{S}_{d+1}$, which is equal to $P(\bld{S})$.  We now have two distinct formulations of the same problem.  The first is to find $\E [ P(\bld{S}(T)) ]$ when $\bld{S}$ solves (\ref{SDE}).  The second is to find $\E [ \tilde{P}(\mathcal{S}(T)) ] = \E [ \mathcal{S}_{d+1}(T) ]$ when $\mathcal{S}$ solves (\ref{modSDE}).  

The second, new formulation has the considerable advantage that its payoff function is linear, and in particular is of the form considered above, so that $\E [P_0]$ may be immediately evaluated using (\ref{P0eval}) with $m = d+1$.  The MLMC method may be applied to (\ref{modSDE}) with $\tilde{P}$ using any discretization available.  This approach will not change the resulting cost scaling, but will reduce the cost by a constant factor that may be significant.  We demonstrate in section 7 via numerical experiments that these savings are frequently considerable.  

This method of using Ito's lemma to linearize the payoff function - we refer to this henceforth as the \textit{Ito linearization technique} - does have two drawbacks.  The first, and most serious, is that two continuous derivatives are required of $P$ for Ito's lemma to apply.  In finance the payoff frequently has a discontinuity in the first derivative - e.g. European options - or even in the function itself - e.g. digital options.  Ito linearization in its present form is not useful for these problems.  

In many applications though, there are many natural payoffs with sufficient regularity.  We have already noted that in chemical kinetics a simple linear payoff function is of interest.  One may also wish to compute the covariances of the chemical concentrations, which may be computed from the means and the payoffs $P(\bld{S}) = S_i S_j$ for each $i,j$, which of course have the necessary smoothness.  

The second drawback is that (\ref{modSDE}) is a $(d+1)$-dimensional system, while (\ref{SDE}) is only $d$-dimensional.  Each time-step of (\ref{modSDE}) is thus slightly more expensive - by a factor of roughly $(d+1)/d$ - than a corresponding time-step of (\ref{SDE}).  In numerical tests, we find that the savings at the base level more than compensate for this added expense.  


\section{Approximate Milstein for MLMC}
We now turn to the derivation of an approximate version of the Milstein discretization that achieves $O(\varepsilon^{-2})$ cost scaling in arbitrary dimension.  There are several observations that make this possible, the first of which is that when estimating $\E [P_l - P_{l-1}]$, the discretizations used to compute $P_l$ and $P_{l-1}$ need not be identical for the reformulated problem (\ref{modSDE}).  

To clarify this point, let us assume we have two discretizations.  Given the same $h$ and $\Delta W$, the `fine' discretization yields the payoff $P^f$ while the `coarse' one yields $P^c$.  We have the following generalization of (\ref{fundMLMCid}):
\begin{equation} \label{GeneralizedSum}
	\E \left[ P^f_L \right] = \E \left[ P^f_0 \right] + \sum_{l=1}^L \left\{ \E \left[ P^f_l - P^c_l \right] + \E \left[ P^c_l - P^f_{l-1} \right] \right\}.
\end{equation}
In the methods of Giles, it is required that 
\begin{equation} \label{GilesConstraint}
	\E \left[ P^c_l \right] = \E [ P^f_{l-1} ]
\end{equation}
for some large class of functions $P$, so that the second term in the sum in (\ref{GeneralizedSum}) is identically zero and (\ref{GeneralizedSum}) reduces to (\ref{fundMLMCid}).  This requires that $\bld{S}^f_{h_l} (T)$ and $\bld{S}^c_{h_l}(T)$ be {\em identically distributed}, which in turn requires that the discretizations used at the fine and coarse levels be at least very nearly identical.   

However, when solving the reformulation afforded by Ito's lemma in the previous section, we may rewrite (\ref{GeneralizedSum}) as 
\begin{equation}
	\E \left[ \mathcal{S}^{f,L}_{d+1} \right] = \E \left[ \mathcal{S}^{f,0}_{d+1} \right] + \sum_{l=1}^L \left\{ \E \left[ \mathcal{S}^{f,l}_{d+1} - \mathcal{S}^{c,l}_{d+1} \right] + \E \left[ \mathcal{S}^{c,l}_{d+1} - \mathcal{S}^{f,l-1}_{d+1} \right] \right\},
\end{equation}
where $\mathcal{S}^{f,l}$ is the result of the `fine' discretization with time-step $h_l$, and similarly for $\mathcal{S}^{c,l}$.  This now reduces to the analogue of (\ref{fundMLMCid}) if
\begin{equation} \label{RicketsonConstraint}
	\E [ \mathcal{S}^{c,l} ] = \E [ \mathcal{S}^{f,l-1} ].
\end{equation}
This condition is actually more than is necessary - we only need the expectations of the last components to match - but there will be no additional difficulty in enforcing this condition.  Being constrained by (\ref{RicketsonConstraint}) instead of (\ref{GilesConstraint}) creates considerable freedom in choosing different fine and coarse discretizations.  We leverage this freedom extensively in the remainder of this section.  

In what follows, we develop an `approximate Milstein' method, whose intended application is MLMC methods applied to the modified SDE (\ref{modSDE}), as it takes advantage of this system's linear payoff function.  We will, however, denote the solution of the SDE by $\bld{S}$ - rather than $\mathcal{S}$ - to emphasize the generality of the specific results.  It is only their application to MLMC that requires the modified SDE.  

We begin by establishing some notation: define
\begin{equation}
	D_i^f(\bld{S}, t, h, \Delta W_n) \equiv a_i(\bld{S}, t) h + \sum_{j=1}^D b_{ij}(\bld{S}, t) \Delta W_{j,n} + \sum_{j,k=1}^D h_{ijk}(\bld{S}, t) (\Delta W_{j,n} \Delta W_{k,n} - \Omega_{jk} h),
\end{equation}
\begin{equation}
	\begin{split}
	&D_i^c(\bld{S}_1, \bld{S}_2, t, h, \delta W_n, \delta W_{n+\frac{1}{2}}) \equiv a_i(\bld{S}_1, t) \Delta t + \sum_{j=1}^D b_{ij}(\bld{S}_2, t) \Delta W_{j,n}\\ 
	&+ \sum_{j,k=1}^D h_{ijk}(\bld{S}_2, t) (\Delta W_{j,n} \Delta W_{k,n} - \Omega_{jk} \Delta t - \delta W_{j,n} \delta W_{k,n+\frac{1}{2}} + \delta W_{j,n+\frac{1}{2}} \delta W_{k,n}),
\end{split}
\end{equation}
where $\Delta W_n$ is a vector in $\R^D$ whose $j^\textrm{th}$ component is $\Delta W_{j,n} = \delta W_{j,n} + \delta W_{j,n+\frac{1}{2}}$.  The analysis is simpler when $M=2$, so we proceed with that case initially and generalize to arbitrary $M$ in section 4.3.  Fix $l$ and set $\delta t = h_l$, $\Delta t = 2 \delta t = h_{l-1}$, $t_n = n\Delta t$.  

\subsection{Review of Antithetic Method}
Because the method we develop here is closely related to the antithetic method of \cite{giles2012antithetic}, we first state and review that algorithm.  In our notation, the antithetic scheme may be written as
\begin{equation} \label{antitheticsummary}
\begin{split}
	\bld{S}_{n+1}^{f,l} = \bld{S}_{n+\frac{1}{2}}^{f,l} + \bld{D}^f(\bld{S}_{n+\frac{1}{2}}^{f,l}, t_{n+\frac{1}{2}}, \delta t, \delta W_{n+\frac{1}{2}}), \quad &\bld{S}_{n+\frac{1}{2}}^{f,l} = \bld{S}_n^{f,l} + \bld{D}^f(\bld{S}_n^{f,l}, t_n, \delta t, \delta W_n) \\
	\bld{S}_{n+1}^{a,l} = \bld{S}_{n+\frac{1}{2}}^{a,l} + \bld{D}^f(\bld{S}_{n+\frac{1}{2}}^{a,l}, t_{n+\frac{1}{2}}, \delta t, \delta W_{n}), \quad &\bld{S}_{n+\frac{1}{2}}^{a,l} = \bld{S}_n^{a,l} + \bld{D}^f(\bld{S}_n^{a,l}, t_n, \delta t, \delta W_{n+\frac{1}{2}})
\end{split}
\end{equation}
where $\bld{D}^f$ is the vector whose $i^\textrm{th}$ component is $D_i^f$, and the fine payoff is set to
\begin{equation} \label{antithetpayoff}
	P^f_l = \frac{1}{2} \left( P(\bld{S}^f_l) + P\left(\bld{S}^a_l\right) \right).
\end{equation}
Meanwhile, the coarse evolution is given by
\begin{equation}
	\bld{S}_{n+1}^{c,l} = \bld{S}_n^{c,l} + \bld{D}^f(\bld{S}^{c,l}_n, t_n, \Delta t, \Delta W_{n}),
\end{equation}
with the coarse payoff set to $P^c_l = P(\bld{S}^c_l)$.  

Notice that the evolution equations for $\bld{S}^f_l$ and $\bld{S}^a_l$ are identical except that the Brownian steps $\delta W_n$ and $\delta W_{n+\frac{1}{2}}$ have been switched.  This has the effect of canceling the leading order contribution of the L\'{e}vy areas when the two are averaged, as in (\ref{antithetpayoff}).  This cancellation makes the $V_l$ scale like $O(h_l^2)$ for twice differentiable payoffs and like $O(h_l^{3/2 - \delta})$ for any $\delta > 0$ when the payoff is Lipschitz, only non-differentiable on a set of measure zero, and the solution is unlikely to be near this set in a certain sense (see \cite{giles2012antithetic} for details).  The scheme thus achieves the $O(\varepsilon^{-2})$ cost scaling in both cases.  

The scheme has two primary drawbacks: 1) it requires twice as much effort to generate $P^f_l$ - due to the need to evolve the antithetic variable $\bld{S}^a_l$ - as an Euler based multilevel scheme, and 2) its derivation in \cite{giles2012antithetic} is restricted to $M=2$.  In our development, we offer a slight improvement to 1) by moving the doubled effort to the coarse level, which is cheaper by a factor of $M$.  Moreover, we generalize both our method and the antithetic method to $M>2$ in sections 4.3 and 5.  

\subsection{Approximate Milstein for $M=2$}
We consider the following pair of schemes for $\bld{S}^f$ and $\bld{S}^c$:
\begin{equation} \label{finedisc}
	\bld{S}_{n+\frac{1}{2}}^{f,l} = \bld{S}_n^{f,l} + \bld{D}^f(\bld{S}_n^{f,l}, t_n, \delta t, \delta W_n),
\end{equation}
\begin{equation} \label{coarsedisc}
	\bld{S}_{n+1}^{c,l} = \bld{S}_n^{c,l} + \bld{D}^c(\bld{S}^{*,l}_n, \bld{S}^{c,l}_n, t_n, \Delta t, \delta W_{n}, \delta W_{n+\frac{1}{2}}),
\end{equation}
where $\bld{S}^{*,l}$ is given by
\begin{equation} \label{stardisc}
	\bld{S}_{n+1}^{*,l} = \bld{S}_n^{*,l} + \bld{D}^f(\bld{S}_n^{*,l}, t_n, \Delta t, \Delta W).
\end{equation}
We set $P^f_l = P(\bld{S}^f_l)$ and $P^c_l = P(\bld{S}^c_l)$.

It is worth clarifying that in this description the number $n$ always indexes the number of level-$l$ \textit{coarse} time steps taken.  This is equal to the number of level-$(l-1)$ \textit{fine} time steps taken, so that number is also indexed by $n$.  By writing (\ref{finedisc}) the way we have, we ensure that $\bld{S}^{f,l}_n$, $\bld{S}^{c,l}_n$, and $\bld{S}^{f,l-1}_n$ are all approximations to $\bld{S}(n\Delta t)$ for each whole number $n$.  In addition to $n=0,1,2,3,...$, we have definitions of $\bld{S}^{f,l}_n$ at $n=1/2,3/2,5/2,...$, but this fact will not concern us.  

In the remainder of this section, we state and prove results that establish first (\ref{RicketsonConstraint}) and then (\ref{Sscale}) with $\beta=2$ for this pair of discretizations.  The more technical proofs are confined to appendices.  
\subsubsection{Equal Expectations}
\begin{theorem}[Equal Expectations] 
\label{equalexp}
For $\bld{S}^f$ and $\bld{S}^c$ as defined in (\ref{finedisc})-(\ref{stardisc}), we have 
\begin{equation}
	\E \left[\bld{S}^{f,l-1}_n \right] = \E \left[\bld{S}^{c,l}_n \right]
\end{equation}
for each $n=0,1,2,3,...$
\end{theorem}
\begin{proof} 
At the $(l-1)^{\textrm{st}}$ level, we have
\begin{equation} \label{finelm1}
	\bld{S}^{f,l-1}_{n+1} = \bld{S}^{f,l-1}_n + \bld{D}^f(\bld{S}_n^{f,l-1}, t_n, \Delta t, \Delta W).
\end{equation}
If we subtract (\ref{coarsedisc}) from (\ref{finelm1}), we find
\begin{equation} \label{long}
	\begin{split}
	\bld{S}^{f,l-1}_{n+1} - \bld{S}^{c,l}_{n+1} &= \bld{S}^{f,l-1}_{n} - \bld{S}^{c,l}_{n} \\
	&+ \left\{ a\left(\bld{S}^{f,l-1}_n\right) - a\left(\bld{S}^{*,l}_n\right) \right\} \Delta t \\
	&+ \sum_{j=1}^D \left\{ b_j\left(\bld{S}^{f,l-1}_n\right) - b_j\left(\bld{S}^{c,l}_n\right) \right\} \left[ \Delta W_{j,n} \right] \\
	&+ \sum_{j,k=1}^D \left\{ h_{jk}\left(\bld{S}^{f,l-1}_n\right) - h_{jk}\left(S^{c,l}_n\right) \right\} \left[ \Delta W_{j,n} \Delta W_{k,n} - \Omega_{jk} \Delta t \right] \\
	&- \sum_{j,k=1}^D \left\{ h_{jk}\left(\bld{S}^{c,l}_n\right) \right\} \left[\delta W_{j,n} \delta W_{k,n+\frac{1}{2}} - \delta W_{j,n+\frac{1}{2}} \delta W_{k,n} \right],
	\end{split}
\end{equation}
where $a$ is the vector whose $i^\textrm{th}$ component is $a_i$, and analogously for $b_j$ and $h_{jk}$.  

We look at (\ref{long}) term by term.  In the last three lines, the term in square brackets has zero expectation and is independent of the term in curly braces - this follows from the fact that each Brownian increment is independent of all those before it.  Therefore, each of these lines has vanishing expectation.  In the second line, $S^{*,l}_n$ and $S^{f,l-1}_n$ are identically distributed for each $n$ because they are approximated by exactly the same method - compare (\ref{finelm1}) and (\ref{stardisc}) - so the term in curly braces has zero expectation.  Therefore, if we take the expectation of (\ref{long}), everything vanishes except the first line.  Thus, we have
\begin{equation} \label{nochange}
	\E \left[ \bld{S}^{f,l-1}_{n+1} \right] - \E \left[ \bld{S}^{c,l}_{n+1} \right] = \E \left[ \bld{S}^{f,l-1}_{n} \right] - \E \left[ \bld{S}^{c,l}_{n} \right].
\end{equation}
Since the coarse and fine approximations start at the same initial condition, the difference in expectation is zero for $n=0$, and (\ref{nochange}) guarantees that this remains the case for all integer $n>0$.  \qquad
\end{proof}
\begin{corollary}
With the same definitions as in Theorem \ref{equalexp}, 
\begin{equation}
	\E \left[ \bld{S}^{*,l}_n \right] = \E \left[ \bld{S}^{c,l}_n \right].
\end{equation}
\end{corollary}
\begin{proof}
Since $\bld{S}^{*,l}$ and $\bld{S}^{f,l-1}$ are identically distributed, they have the same expectation, so this follows directly from Theorem \ref{equalexp}.  \qquad
\end{proof}
\subsubsection{Variance Scaling}
Before establishing the variance scaling (\ref{Sscale}) required for the $\varepsilon^{-2}$ cost scaling, we need three lemmas.  The first establishes that the weak difference between coarse and starred approximations is $O(\Delta t)$, while the last two are convenient rewritings of the fine and coarse discretizations.  
\begin{lemma} \label{weakstar}
The weak difference between $\bld{S}^{c,l}$ and $\bld{S}^{*,l}$ is $O(\Delta t)$.  That is, for sufficiently differentiable $f: \R^d \rightarrow \R$, we have
\begin{equation}
	\left| \E \left[ f(\bld{S}^{c,l}_n) \right] - \E \left[ f(\bld{S}^{*,l}_n) \right] \right| = O(\Delta t)
\end{equation}
for all $n\leq T/\Delta t$.  
\end{lemma}
\begin{proof}
See appendix A. \qquad
\end{proof}

In the proof of lemma \ref{weakstar} above, we require that $f$ have four continuous and bounded derivatives.  Elsewhere in our development, the payoff and SDE coefficients are only required to possess two derivatives, and in practice the scaling predicted by lemma \ref{weakstar} is observed in these cases as well.  It seems likely that by following \cite{kloeden2011numerical}, lemma \ref{weakstar} could be reestablished for $f$ merely H\"{o}lder continuous, but such an exercise is beyond the scope of this paper.  
\begin{lemma} \label{fineM2}
The fine discretization (\ref{finedisc}) can be rewritten as
\begin{equation}
	S^{f,l}_{i,n+1} = S^{f,l}_{i,n} + D_i^c(\bld{S}^{f,l}_n, \bld{S}^{f,l}_n, t_n, \Delta t, \delta W_n, \delta W_{n+\frac{1}{2}}) + M^f_{i,n} + N^f_{i,n}
\end{equation}
where $\E[M^f_{i,n}] = 0$ and
\begin{equation}
	\E \left[ \max_{n\leq N} \norm{M_n^f}^p \right] = O\left( \Delta t^{3p/2} \right), \qquad \E \left[ \max_{n\leq N} \norm{N_n^f}^p \right] = O \left( \Delta t^{2p} \right)
\end{equation}
for any integer $p \geq 2$.
\end{lemma}
\begin{proof}
The lemma and proof are identical to Lemma 4.7 and its proof in \cite{giles2012antithetic}. \qquad
\end{proof}
\begin{lemma} \label{coarseM2}
The coarse discretization (\ref{coarsedisc}) may be rewritten as
\begin{equation}
	S^{c,l}_{i,n+1} = S^{c,l}_{i,n} + D_i^c(\bld{S}^{c,l}_n, \bld{S}^{c,l}_n, t_n, \Delta t, \delta W_n, \delta W_{n+\frac{1}{2}}) + M^c_{i,n} + N^c_{i,n}
\end{equation}
where $\E[M^c_{i,n}] = 0$ and
\begin{equation}
	\E \left[ \max_{n\leq N} \norm{M_n^c}^p \right] = O\left( \Delta t^{3p/2} \right), \qquad \E \left[ \max_{n\leq N} \norm{N_n^c}^p \right] = O \left( \Delta t^{2p} \right).
\end{equation}
\end{lemma}
\begin{proof}
Simple algebra shows that 
\begin{equation}
\begin{split}
	S^{c,l}_{i,n+1} = S^{c,l}_{i,n} &+ D_i^c(\bld{S}^{c,l}_n, \bld{S}^{c,l}_n, t_n, \Delta t, \delta W_n, \delta W_{n+\frac{1}{2}}) \\
	&+ \left[ a_i\left(\bld{S}^{*,l}_n\right) - a_i\left(\bld{S}^{c,l}_n\right) \right] \Delta t, 
\end{split}
\end{equation}
so the lemma reduces to analyzing the second line of this expression.  Define $\Delta a_{i,n} = a_i\left(\bld{S}^{*,l}_n\right) - a_i\left(\bld{S}^{c,l}_n\right)$, and write the term in question as
\begin{equation} \label{DeltaA}
	\Delta a_{i,n} \Delta t = \E \left[ \Delta a_{i,n} \right] \Delta t + \left\{ \Delta a_{i,n} - \E \left[ \Delta a_{i,n} \right] \right\} \Delta t.
\end{equation}
By Lemma \ref{weakstar} above (which we again note uses four derivatives), we have $\E [\Delta a_{i,n}] = O(\Delta t)$, so that the first term is $O(\Delta t^2)$.  Thus, we define $N^c_{i,n} = \E [\Delta a_{i,n}] \Delta t$.  

The second term on the right of (\ref{DeltaA}) clearly has zero expectation.  The first term in the curly braces is $O(\sqrt{\Delta t})$ by strong convergence of both schemes (\ref{coarsedisc}) and (\ref{stardisc}), and the second term in the curly braces is $O(\Delta t)$ as before, so their difference is $O(\sqrt{\Delta t})$.  Thus, the second term on the right of (\ref{DeltaA}) is $O(\Delta t^{3/2})$, so we define it to be $M^c_{i,n}$.  \qquad
\end{proof}

Finally, we are ready to prove the desired scaling of the variances:
\begin{theorem}[Variance Scaling] \label{varscale}
Assume the $a_i$ have four continuous bounded derivatives, $b_{ij}$ are twice continuously differentiable with both derivatives uniformly bounded, and that the $h_{ijk}$ have uniformly bounded first derivative.  Then, for the pair of fine and coarse discretizations (\ref{finedisc})-(\ref{stardisc}), we have (\ref{Sscale}) with $\beta = 2$.  In fact, we have the stronger statement
\begin{equation}
	\E \left[ \max_{n\leq N} \norm{\bld{S}^{f,l}_n - \bld{S}^{c,l}_n}^2 \right] = O \left( \Delta t^2 \right),
\end{equation}
where $N = T/\Delta t$.  
\end{theorem}
\begin{proof}
See appendix B. \qquad
\end{proof}

\subsection{Generalization to $M>2$}
Thus far, we've developed the approximate Milstein method assuming that the difference in time step at adjacent levels (refinement factor) $M$ is equal to two.  This assumption is also made in Giles' development of the antithetic method.  However, in \cite{giles2008multilevel} Giles argues that the optimal refinement factor is near seven for an Euler-based multilevel scheme, and a similar argument for Milstein shows the optimal choice to be near four.

Following \cite{giles2008multilevel}, the latter argument proceeds as follows:  Noting that $P^f_l - P^c_l = (P^f_l - P) - (P^c_l - P)$, where $P$ is the exact mean payoff, we infer that
\begin{equation}
	(M - 1)^2 k h_l^2 \leq V_l \leq (M + 1)^2 k h_l^2,
\end{equation}
for some constant $k$, where the lower and upper bounds correspond to perfect correlation and anti-correlation between $P^f_l - P$ and $P^c_l - P$.  Supposing for simplicity that the actual variance is approximately the geometric mean of the two extremes, we have
\begin{equation}
	V_l \approx (M^2 - 1) k h_l^2.
\end{equation}
Substituting this expression into the cost formula (\ref{costformula}), we have (ignoring for clarity the fact that $V_0$ need not obey any scaling law)
\begin{equation}
	K \propto \varepsilon^{-2} \frac{(M^2 - 1)(1 + M^{-1})}{\left(\sqrt{M} - 1\right)^2},
\end{equation}
which for fixed $\varepsilon$ has its minimum for $M$ between $4$ and $5$.  There is thus motivation to study arbitrary $M$.  

Notationally, moving to arbitrary $M$ changes (\ref{finedisc}) to read
\begin{equation} \label{newfinedisc}
	\bld{S}_{n+\frac{1}{M}}^{f,l} = \bld{S}_n^{f,l} + \bld{D}^f(\bld{S}_n^{f,l}, t_n,\delta t, \delta W_{j,n}),
\end{equation}
and we set 
\begin{equation}
	\Delta t = M \delta t, \qquad \Delta W_{j,n} = \sum_{m=0}^{M-1} \delta W_{j,n+\frac{m}{M}}.
\end{equation}
To see how to change (\ref{coarsedisc}), we present the following generalization of Lemma \ref{fineM2}:
\begin{lemma} \label{arbM}
The fine discretization (\ref{newfinedisc}) can be rewritten as
\begin{equation}
\begin{split}
	S^{f,l}_{i,n+1} &= S^{f,l}_{i,n} + D_i^f(\bld{S}^{f,l}_n, t_n, \Delta t, \Delta W_n,)  \\
	&- \sum_{j,k=1}^D h_{ijk}\left(\bld{S}^{f,l}_n\right) \left( \mathcal{A}_{jk,n} - \mathcal{A}_{kj,n} \right) \\ 
	&+ M^f_{i,n} + N^f_{i,n},
\end{split}
\end{equation}
where $\E[M^f_{i,n}] = 0$ and
\begin{equation}
	\E \left[ \max_{n\leq N} \norm{M_n^f}^p \right] = O\left( \Delta t^{3p/2} \right), \qquad \E \left[ \max_{n\leq N} \norm{N_n^f}^p \right] = O \left( \Delta t^{2p} \right),
\end{equation}
and $\mathcal{A}_{jk,n}$ is defined by 
\begin{equation}
	\mathcal{A}_{jk,n} \equiv \sum_{m=1}^{M-1} \left( \delta W_{k,n+\frac{m}{M}} \sum_{q=0}^{m-1} \delta W_{j,n+\frac{q}{M}} \right).
\end{equation}
\end{lemma}
\begin{proof}
See appendix C. \qquad
\end{proof}

We note that, as described in \cite{dimits2013higher}, $(\mathcal{A}_{jk,n} - \mathcal{A}_{kj,n})$ is a quadrature scheme for the Levy area $A_{jk,n}$ obtained by dividing up the time step $\Delta t$ into $M$ equal parts.  As noted in \cite{dimits2013higher}, computing the $\mathcal{A}_{jk,n}$ as written can be done in $O(M)$ operations, even though the double sum contains $O(M^2)$ terms.  

With Lemma \ref{arbM} in hand, we can generalize the coarse discretization to arbitrary $M$.  We define
\begin{equation}
	\begin{split}
	&D_i^{c,M}\left(\bld{S}_1, \bld{S}_2, t, \Delta t, \delta W_{j,n+\frac{m}{M}}\right) \equiv a_i(\bld{S}_1,t) \Delta t + \sum_{j=1}^D b_{ij}(\bld{S}_2,t) \Delta W_{j,n}\\ 
	&+ \sum_{j,k=1}^D h_{ijk}(\bld{S}_2,t) (\Delta W_{j,n} \Delta W_{k,n} - \Omega_{jk} \Delta t - \mathcal{A}_{jk,n} + \mathcal{A}_{kj,n}).
\end{split}
\end{equation}
Meanwhile, the starred discretization (\ref{stardisc}) remains unchanged.  With these definitions and Lemma \ref{arbM} in place of Lemma \ref{fineM2}, the proofs of Theorems \ref{equalexp} and \ref{varscale} generalize to arbitrary $M$ with only straightforward notational changes.  

\section{Generalization of Antithetic Method to $M>2$}
In this section, we demonstrate that the antithetic method (\ref{antitheticsummary}) may be straightforwardly generalized to arbitrary $M$.  In particular, the same variance scaling can be achieved by using an antithetic variable for which the order of the $M$ Brownian fine sub-steps of each coarse Brownian step is completely reversed.  

More explicitly, set $\bar{m} = (M-1) - m$, then we rewrite (\ref{antitheticsummary}) as
\begin{equation}
\begin{split}
	\bld{S}_{n+\frac{m+1}{M}}^{f,l} &= \bld{S}_{n+\frac{m}{M}}^{f,l} + \bld{D}^f(\bld{S}_{n+\frac{m}{M}}^{f,l}, t_{n+\frac{m}{M}}, \delta t, \delta W_{n+m/M}), \\
	\bld{S}_{n+\frac{m+1}{M}}^{a,l} &= \bld{S}_{n+\frac{m}{M}}^{a,l} + \bld{D}^f(\bld{S}_{n+\frac{m}{M}}^{a,l}, t_{n+\frac{m}{M}}, \delta t, \delta W_{n + \bar{m}/M}),
\end{split}
\end{equation}
for each $m = 0, 1, 2, ..., M-1$.  As before, $\bld{S}^{f,l}$ and $\bld{S}^{a,l}$ are identical except that the Brownian motions are indexed by $\bar{m}$ rather than $m$ for the antithetic variable.  By applying Lemma \ref{arbM} to the antithetic variable $\bld{S}^{a,l}$, we find that its discretization can be rewritten as 
\begin{equation}
\begin{split}
	S^{a,l}_{i,n+1} &= S^{a,l}_{i,n} + D_i^f(\bld{S}^{a,l}_n, t_n, \Delta t, \Delta W_n,)  \\
	&- \sum_{j,k=1}^D h_{ijk}\left(\bld{S}^{a,l}_n\right) \left( \bar{\mathcal{A}}_{jk,n} - \bar{\mathcal{A}}_{kj,n} \right) \\ 
	&+ M^a_{i,n} + N^a_{i,n},
\end{split}
\end{equation}
where $M^a_{i,n}$ and $N^a_{i,n}$ obey the same scalings as $M^f_{i,n}$ and $N^f_{i,n}$.  The quantity $\bar{\mathcal{A}}_{jk,n}$ is defined by
\begin{equation}
	\bar{\mathcal{A}}_{jk,n} \equiv \sum_{m=1}^{M-1} \left( \delta W_{k,n+\bar{m}/M} \sum_{q=0}^{m-1} \delta W_{j,n+\bar{q}/M} \right),
\end{equation}
with $\bar{m}$ as defined before and $\bar{q} = (M-1) - q$.  The improved variance scaling then results from the following lemma, originally published in \cite{giles2012antithetic2} with different notation and reprinted here for clarity.
\begin{lemma} \label{antithet}
\begin{equation}
	\bar{\mathcal{A}}_{jk,n} = \mathcal{A}_{kj,n}.
\end{equation}
\end{lemma}
\begin{proof}
The proof amounts to a computation:
\begin{align}
	\bar{\mathcal{A}}_{jk,n} &= \sum_{m=1}^{M-1} \sum_{q=0}^{m-1} \delta W_{k,n+\bar{m}/M}  \delta W_{j,n+\bar{q}/M}  \\
	\label{line2} &= \sum_{q=0}^{M-2} \sum_{m=q+1}^{M-1} \delta W_{k,n+\bar{m}/M}  \delta W_{j,n+\bar{q}/M} \\
	\label{line3} &= \sum_{\bar{q}=1}^{M-1} \sum_{\bar{m}=0}^{\bar{q}-1} \delta W_{k,n+\bar{m}/M}  \delta W_{j,n+\bar{q}/M}\\
	&= \mathcal{A}_{kj,n}.
\end{align}
Notice that (\ref{line2}) comes from simply switching the order of summation, while (\ref{line3}) results from rewriting the sums in terms of $\bar{m}$ and $\bar{q}$. \qquad
\end{proof}

This lemma implies that, if we define $\bar{\bld{S}}^{f,l} = \frac{1}{2} (\bld{S}^{f,l} + \bld{S}^{a,l})$, we can write
\begin{equation}
\begin{split}
	\bar{S}^{f,l}_{i,n+1} &= \bar{S}^{f,l}_{i,n} + D_i^f(\bar{\bld{S}}^{f,l}_n, t_n, \Delta t, \Delta W_n,)  \\
	&+ M_{i,n} + N_{i,n},
\end{split}
\end{equation}
where $M_{i,n}$ and $N_{i,n}$ obey the same scalings as usual.  This is in direct analogue to Lemma 4.9 in \cite{giles2012antithetic}, and its proof is identical, so we omit it.  This, in turn, allows one to show the analogue of Theorem 4.10 in \cite{giles2012antithetic} and Theorem \ref{varscale} in the present work:
\begin{equation} \label{410}
	\E \left[ \max_{n\leq N} \norm{\bar{\bld{S}}^{f,l}_n - \bld{S}^{c,l}_n}^p \right] = O \left( \Delta t^p \right)
\end{equation}
for each $p \geq 2$.  The desired variance scaling finally follows directly from Lemma 2.2 in \cite{giles2012antithetic}, which we restate without proof here:
\begin{equation}
	\E \left[ \left( \frac{1}{2} \left( P(\bld{S}^{f}) + P(\bld{S}^{a}) \right) - P(\bld{S}^{c}) \right)^p \right] \lesssim \E \left[ \norm{\bar{\bld{S}}^{f} - \bld{S}^{c}}^p \right] + \E \left[ \norm{\bld{S}^{f} - \bld{S}^{a}}^{2p} \right],
\end{equation}
where we've omitted the $l$ superscripts and assumed that $P$ had two continuous and bounded derivatives.  The first term on the right is $O(\Delta t^p)$ by (\ref{410}), and the second term has the same scaling by strong convergence: indeed, $\bld{S}^f - \bld{S}^a = (\bld{S}^f - \bld{S}^c) - (\bld{S}^a - \bld{S}^c)$, and each of the latter two terms is $O(\Delta t^{1/2})$ by the strong convergence of the discretization.  

Finally, we note that Theorem 5.2 in \cite{giles2012antithetic} may be applied - unchanged - to these results, demonstrating that $V_l = O(h_l^{3/2 - \delta})$ for any $\delta >0$ when $P$ is merely Lipschitz - so long as the set $A$ on which $P$ is non-differentiable is measure zero and
\begin{equation}
	\mathbb{P} \left( \min_{y\in A} \norm{\bld{S}(T) - y} \leq \varepsilon \right) \leq c \, \varepsilon
\end{equation}
for some $c > 0$ and for all $\varepsilon > 0$.  

This completes the generalization of the antithetic method to arbitrary $M$.  


\section{Summary and Implementation}
We present an outline of the modified MLMC method including the Ito linearization at the lowest level.  This is to be compared to the analogous algorithm in \cite{giles2008multilevel}.  This may be used with any discretization we choose - including the approximate Milstein method introduced in section 4 and the generalized antithetic method in section 5 - so long as the discretization is first order in the weak sense.  We denote by $\beta$ the expected scaling of the $V_l$.  That is, $\beta = 2$ for approximate Milstein or generalized antithetic (assuming $P$ has the necessary regularity), and $\beta = 1$ for Euler.  
\begin{enumerate}
	\item Set 
	\begin{equation}
		\E [P_0] = P(\bld{S}_0) + \alpha_{d+1}(\bld{S}_0) T,
	\end{equation}
	where $\alpha_{d+1}$ is as defined in (\ref{abdefs}).  
	\item Begin with $L=1$.
	\item Estimate $V_L$ and $\hat{Y}_L$ using an initial $N_L^i$ samples, defined by
	\begin{equation} \label{initialsamples}
		N_L^i = \left\{ 
		\begin{array}{lr}
       			400 & : L = 1 \\
       			M^{-(\beta + 1)/2}N_{L-1} & : L > 1
     		\end{array}
   		\right.
	\end{equation}
	\item Set $N_l$ according to
	\begin{equation} \label{optsamp2}
		N_l = \left\lceil \frac{2}{\varepsilon^2} \sqrt{V_l h_l} \left( \sum_{l=1}^{L} \sqrt{V_l/h_l} \right) \right\rceil
	\end{equation}
	for each $l = 1, 2, ..., L$, as per (\ref{optimalsamplenumbers}).  
	\item Generate additional samples at each level as needed for new $N_l$, using discretization of your choice to generate approximate solutions of (\ref{modSDE}).  Use these samples to update the estimates of the $V_l$ and $\hat{Y}_l$. 
	\item If $L < 2$ or 
	\begin{equation} \label{convtest}
		\max \left\{ \left| \hat{Y}_L \right|, M^{-1} \left| \hat{Y}_{L-1} \right| \right\} > \frac{\varepsilon}{\sqrt{2}},
	\end{equation}
	let $L \rightarrow L+1$ and go to step 3.  Else, end with payoff estimate of
	\begin{equation}
		\hat{P}_L = \E [P_0] + \sum_{l=1}^L \hat{Y}_l.
	\end{equation}
\end{enumerate}

The inequality (\ref{convtest}) is the convergence criterion used in \cite{giles2008multilevel}.  It determines the algorithm to be converged if the bias error is estimated to be at most $\varepsilon/\sqrt{2}$ when using either of the two finest levels in the estimation.

Equation (\ref{initialsamples}) in step 3 is worthy of elaboration.  When $L=1$, we have no information about how many samples we might expect to need, so we pick an arbitrary large number - we find that $400$ works well in our test cases, but the number will be problem dependent.  However, for $L>1$, the expected scaling of the $V_l$ allows us to estimate the number of samples needed at the $L^\textrm{th}$ level, using
\begin{equation}
	N_L \propto \sqrt{V_L h_L} = M^{-(\beta + 1)/2} \sqrt{V_{L-1} h_{L-1}} \propto M^{-(\beta + 1)/2} N_{L-1}.
\end{equation}
Particularly at large $L$, when there is relatively little change to the sum in (\ref{optsamp2}) as a result of incrementing $L$, (\ref{initialsamples}) is thus a good estimate of $N_L$ as it will be set in step 4.  This is preferable to the technique used in \cite{giles2008multilevel} - where $N_L^i = 10^4$, regardless of $L$ - because we avoid wasteful sampling at the high levels where it may be the case that $N_l \ll 10^4$.


\section{Numerical Results}
For our numerical tests, we apply our methods to the Heston model - a financial stochastic volatility model \cite{heston1993closed} - given by
\begin{equation} \label{heston}
\begin{split}
	dS_1 &= \kappa (\theta - S_1) \, dt + \xi \sqrt{S_1} \, dW_1, \\
	dS_2 &= \mu S_2 \, dt + \eta \sqrt{S_1} S_2 \, dW_2,
\end{split}
\end{equation}
where $S_1$ represents the volatility and $S_2$ the asset price.  Throughout our tests, we set the constants $\theta = \mu = \xi = \kappa = 1$ and $\eta = 1/4$.  We find that this choice of constants allows us to conduct tests with relatively large $L$, where the benefits of MLMC are most obvious.  We set $\bld{S}_0 = (0.5, 1)$, $\Omega_{jk} = \delta_{jk}$, and $T = 0.125$ - a short time simulation allows us to push the limits of the accuracy of the method.  Notice that, for this system, $h_{221} = \eta S_2 / 4$, so that the Milstein discretization does in fact feature L\'{e}vy areas.  

Notice also that the coefficients of this SDE system do not have uniformly bounded derivatives - namely, the $b$'s have divergent derivative at $S_1 = 0$ - so the assumptions for all of the foregoing results do not hold (see e.g.\ theorem \ref{varscale}), nor do those for standard weak convergence results.  However, we have constructed the system so that $S_1$ is extremely unlikely to approach zero, so that in practice all derivatives are essentially bounded.  Indeed, we find excellent agreement between the theory and numerical results.  

We test several distinct new MLMC variants:
\begin{enumerate}
	\item Generalized antithetic method for a payoff with discontinuous derivative
	\item Ito linearization technique for smooth payoff using:
	\begin{enumerate}
		\item Euler discretization
		\item Approximate Milstein discretization
		\item Generalized antithetic discretization
	\end{enumerate}
\end{enumerate}
Each of these is compared to the original Euler-based MLMC method introduced in \cite{giles2008multilevel} and the original antithetic method proposed in \cite{giles2012antithetic}.  

In all tests, the computational cost is estimated by the total number of time steps taken, weighted by the dimension of the system being solved.  That is, for the standard Euler MLMC method, we set
\begin{equation} \label{ecost}
	K^e = \sum_{l=0}^L N_l(h_l^{-1} + h_{l-1}^{-1}), 
\end{equation}
while for the antithetic and generalized antithetic methods (without Ito linearization), we set
\begin{equation} \label{acost}
	K^a = \sum_{l=0}^L N_l (2 h_l^{-1} + h_{l-1}^{-1})
\end{equation}
to account for the added computation of the antithetic variable $\bld{S}^{a,l}$.  For the approximate Milstein method, we set
\begin{equation}
	K^m = \frac{d+1}{d} \sum_{l=0}^L N_l(h_l^{-1} + 2 h_{l-1}^{-1}),
\end{equation}
accounting both for the added cost of computing $\bld{S}^{*,l}$ and the extra dimension.  When Ito linearization is applied to Euler and antithetic methods, we simply scale (\ref{ecost}) and (\ref{acost}) by the factor $(d+1)/d$ and note that $N_0 = 1$.  

The counting of time steps is the standard method of estimating computational complexity in the MLMC literature - it is used in \cite{giles2008multilevel, giles2012antithetic}, among others.  However, we note in our discussion at the end of this section that there are other relevant considerations as well.  

\subsection{Generalized antithetic test}
We use a standard European option as the payoff function:
\begin{equation}
	P(\bld{S}) = \max \left\{ 0, S_2(T) - S_2(0) \right\}.
\end{equation}
Figure 2 demonstrates both the improved variance scaling (left) and corresponding reduction in computational cost (right) afforded by antithetic methods.  As predicted in section 4.3, $M$ equal to 4 or 5 minimizes the computational cost.  The reduction in cost gained by the generalization to arbitrary $M$ is comparable to that gained by moving from Euler to the original antithetic method.  We note that the cost is reduced for larger $M$ in spite of the fact that increasing $M$ actually increases the individual $V_l$.  This because at large $M$, we need fewer levels, so the sum of the variances is smaller because there are fewer terms in the sum.  
\begin{figure}
  \centering
	\includegraphics[width=.47\textwidth, height=.38\textwidth]{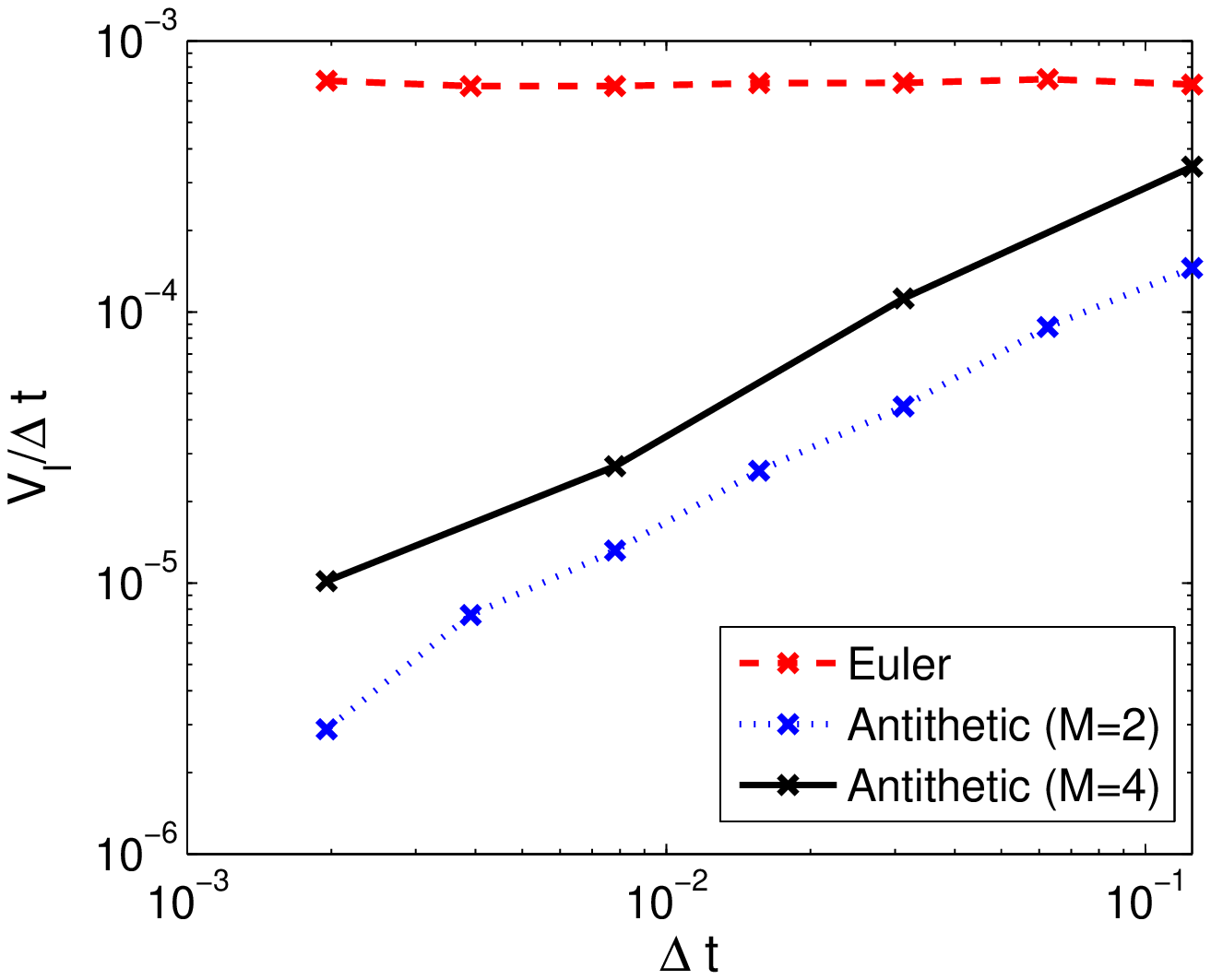}
	\includegraphics[width=.47\textwidth, height=.38\textwidth]{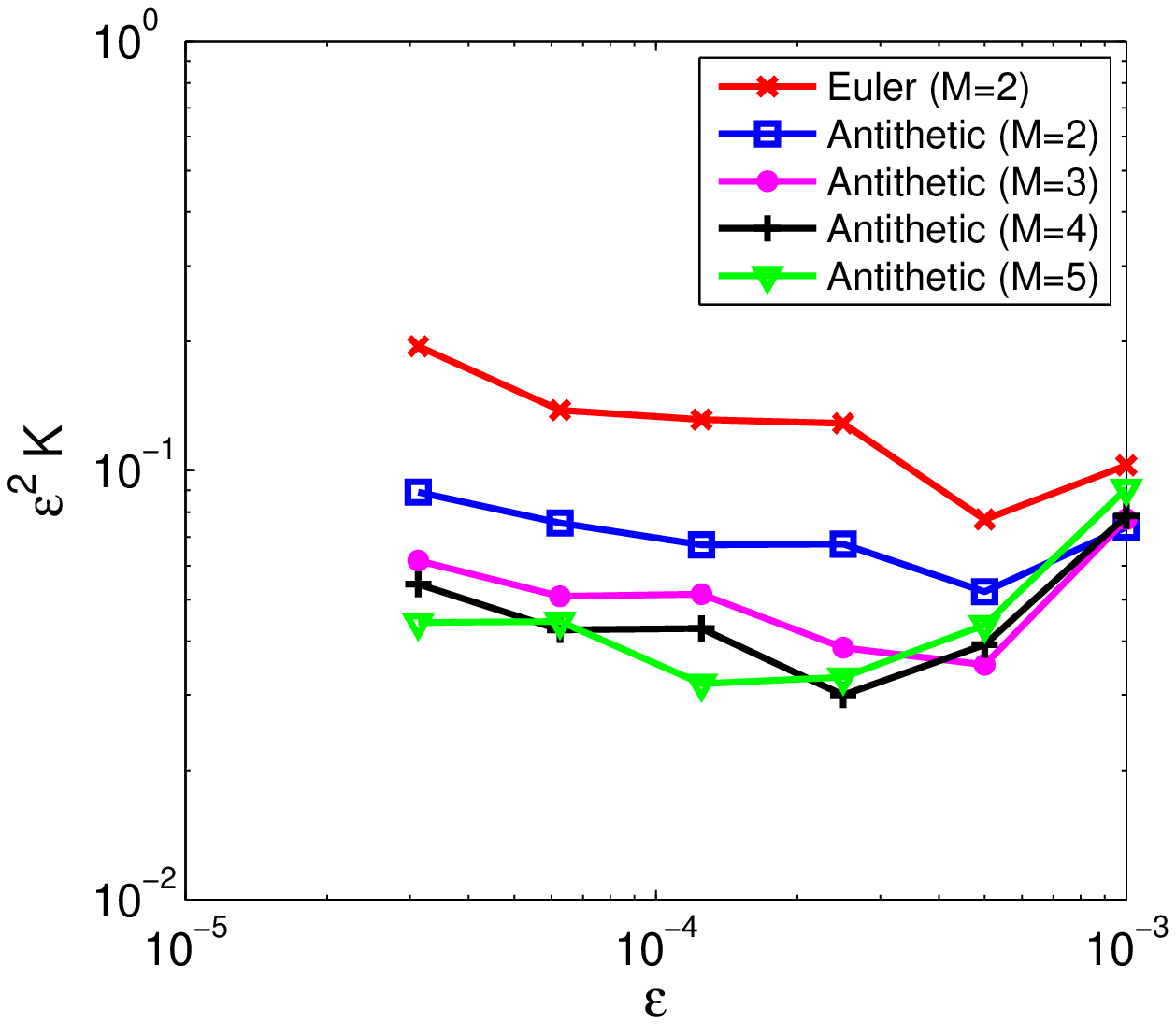}
	\caption{\underline{Left}: A plot of $V_l / \Delta t$ against $\Delta t$ for the Euler, standard antithetic, and generalized antithetic methods.  As expected, the Euler curve is constant, while both antithetic methods demonstrate the same scaling.  \underline{Right}: A comparison of the computational cost of the Euler method to the antithetic method for various refinement factors $M$ as a function of the error tolerance $\varepsilon$.  As expected, an $M$ of $4$ or $5$ is optimal.}
\end{figure} 

\subsection{Ito linearization and approximate Milstein test}
We again use the Heston model with the same parameters but change the payoff function to 
\begin{equation}
	P(\bld{S}) = \sin S_2,
\end{equation}
which of course has the requisite regularity.  Figure 3 shows the results, plotting computational cost against error tolerance $\varepsilon$.  The Euler, antithetic, and approximate Milstein discretizations are all plotted, each with and without Ito linearization.  

For each discretization, Ito linearization improves the efficiency of the scheme, by an order of magnitude in some cases.  Note that in the case of approximate Milstein, the use of Ito's lemma is required for the expectations at adjacent levels to match - see the beginning of section 4 for details - so that the exclusion of Ito linearization is somewhat artificial.  This accounts for the disproportionately large expense when Ito linearization is excluded, since the use of Ito's lemma increases the dimension of the system we solve.  

The most efficient scheme of those tested is approximate Milstein with Ito linearization, although the advantage over generalized antithetic with Ito linearization is relatively small.  As expected based on fig.\ 1, the approximate Milstein and antithetic methods benefit more from Ito linearization than does Euler because a larger fraction of the work is concentrated at the base level.  

\begin{figure}
 \centering
 	\includegraphics[width=.6\textwidth]{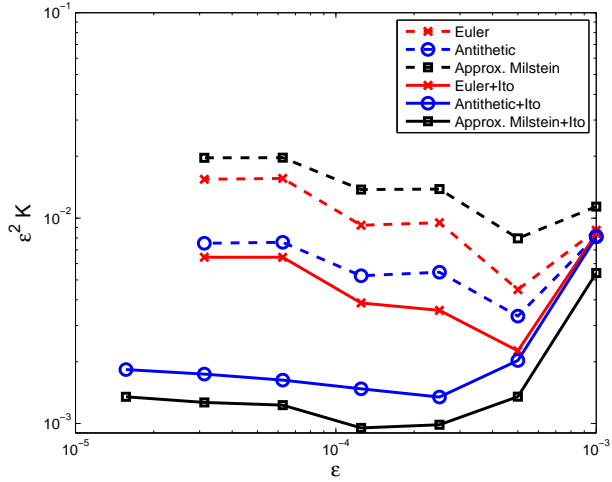}
	\caption{Computational cost of MLMC variants with sinusoidal payoff as a function of error tolerance.  Euler (red), antithetic (blue), and approximate Milstein (black) discretizations are shown, each with Ito linearization (solid) and without (dashed).  All results in this plot use $M=4$.}
\end{figure}

We note that neither the approximate Milstein nor the antithetic method produces a completely flat curve when we plot $\varepsilon^2 K$ against $\varepsilon$, as would be predicted by the asymptotic analysis.  However, the observations are perfectly consistent with (\ref{costformula}): the finite sum is bounded by an infinite sum, but is not itself constant in $\varepsilon$.  We in fact only expect a flat curve as $L \rightarrow \infty$, and all of the tests conducted here have $L \leq 7$.  

Finally, we note that in many cases we get a better-than-expected cost scaling for large $\varepsilon$.  This is a result of setting a fixed number of initial samples for $L=1$ in (\ref{initialsamples}) - and for $L=0$ when Ito linearization is not used - which can turn out to be more than is actually necessary for large $\varepsilon$, making the cost artificially large.  We leave this effect in the plots because it is a practical reality of MLMC.  

\subsection{Discussion}
The techniques introduced in the present work have the common aim of optimizing MLMC simulations of SDEs.  Since the standard MLMC algorithm with the Euler discretization already achieves a nearly optimal cost-to-error scaling, the savings are relatively modest - we rarely save more than a single order of magnitude.  While these improvements are hardly negligible, they are small enough that some care in analyzing all sources of computational cost and coding optimization is justified.  

In particular, the estimation of computational cost by the total number of time steps taken ignores the fact that not all time steps have identical complexity.  In the antithetic and approximate Milstein discretizations, there is an additional term to be computed at each time step.  The dominant cost turns out to be the computation of the rank-3 tensor $h_{ijk}$, which requires $O(d^2D^2)$ operations - the tensor has $dD^2$ elements, and the computation of each requires a sum of $d$ terms.  In contrast, the dominant computation in an Euler time step is the matrix-vector product $\sum_j b_{ij} W_j$, which is $O(dD)$.  

With this in mind, a fairer estimate of the computational cost for each method is
\begin{equation}
	K = \left\{
	\begin{array}{ll}
       		O\left(\varepsilon^{-2} (\log \varepsilon)^2 d D\right) & : \textrm{Euler} \\
       		O\left( \varepsilon^{-2} d^2 D^2\right) & : \textrm{Anithetic \& Approx.\ Milst.}
     	\end{array}
   	\right.
\end{equation}¥
Thus, the optimal discretization is problem dependent - for any fixed $\varepsilon$, Euler with be optimal for some sufficiently large value of $dD$, while antithetic/approximate Milstein is optimal for smaller values of $dD$.  In financial and chemical kinetic applications, $d$ and $D$ are frequently large - very possibly exceeding $\log \varepsilon$.  

This situation is, however, the worst case.  Often, $b_{ij}$ and/or $h_{ijk}$ exhibit some form of sparsity, which the code may be written to exploit.  It may be possible to leverage the relationship between $b_{ij}$ and $h_{ijk}$ to further accelerate their computation.  This is an interesting area of future research.  


\section{Conclusions}
In this paper we have introduced three related improvements to MLMC methods for SDEs.  First, we have introduced the idea of Ito linearization, which makes the computation of the base level payoff essentially free, at the price of increasing the dimension of the SDE by one.  Secondly, we have introduced an approximate Milstein discretization which, in conjunction with Ito linearization, achieves an $O(\varepsilon^{-2})$ cost scaling with slightly reduced cost compared to the antithetic method.  Finally, we demonstrated that the antithetic method can be generalized to arbitrary $M$ without introducing any additional antithetic paths.  

The first two techniques are applicable only to payoff functions with two continuous derivatives.  As such, they are of very limited use in financial applications, but are expected to be applicable to other fields - examples from chemical kinetics have been cited in the text.  The generalized antithetic method, however, requires only a Lipschitz, piecewise smooth payoff, and may thus find applications in finance as well as other disciplines.  

Each new method has been tested on a simple SDE system, and we find excellent agreement between the analysis and the numerical results.  In the cases in which they are applicable, our new methods consistently outperform the present state-of-the-art.  

\section*{Acknowledgements}
The author is grateful to Russel Caflisch for numerous discussions that were instrumental in the development of the methods presented here.  Also acknowledged are many helpful conversations with Mark Rosin regarding the theory and implementation of MLMC schemes in general.  Additional thanks go to Andris Dimits, Bruce Cohen and Michael Giles.  


\appendix 

\section{Proof of Lemma \ref{weakstar}}
The argument here follows that in \cite{mil1979method}.  It proceeds in two stages: first, we express the total error as a sum of various local truncation errors, totaling $O(\Delta t^{-1})$ in number.  Second, we show that each local truncation error is $O(\Delta t^2)$.  

Toward the first end, we introduce some notation for this proof not used elsewhere: Denote by $\bld{S}^{c,l}_n[\bld{x}]$ that solution of the recursion equation (\ref{coarsedisc}) at time $t_n$ which starts at $\bld{x}$ at time zero, and similarly for $\bld{S}^{*,l}_n [\bld{x}]$.  For this proof we will assume that the system is autonomous, so that
\begin{equation} \label{id1}
	\bld{S}^{c,l}_{n+1}[\bld{x}] = \bld{S}^{c,l}_1\left[\bld{S}^{c,l}_n[\bld{x}] \right],
\end{equation}
and similarly for $\bld{S}^{*,l}_{n+1}[\bld{x}]$.  All of the arguments presented here generalize to the non-autonomous case, but the notation is much cleaner if autonomy is assumed.  Define $g^c_n(\bld{x}) = \E f(\bld{S}^{c,l}_n[\bld{x}])$ and similarly for $g^*_n(\bld{x})$.  Then, leveraging (\ref{id1}), we have
\begin{equation} \label{err1}
\begin{split}
	g^c_{n+1}(\bld{x}) - g^*_{n+1}(\bld{x}) &= \E f\left(\bld{S}^{c,l}_1\left[\bld{S}^{c,l}_n[\bld{x}] \right] \right) - \E f\left(\bld{S}^{*,l}_1\left[\bld{S}^{*,l}_n[\bld{x}] \right]\right) \\
	&= \E g^c_1 \left( \bld{S}^{c,l}_n[\bld{x}] \right) - \E g^*_1\left( \bld{S}^{*,l}_n[\bld{x}] \right) \\
	&= \E \left\{ g^c_1 \left( \bld{S}^{c,l}_n[\bld{x}] \right) - g^*_1\left( \bld{S}^{c,l}_n[\bld{x}] \right) \right\} 
	+ \E \left\{ g_1^* \left( \bld{S}^{c,l}_n[\bld{x}] \right) - g_1^* \left( \bld{S}^{*,l}_n[\bld{x}] \right) \right\}.
\end{split}
\end{equation}
The first expectation in the third line is a local truncation error: it's the difference in $f$ evaluated at the coarse and starred discretizations after one time step, when both started at the same place; namely, $\bld{S}^{c,l}_n[\bld{x}]$.  Due to the nature of the coarse and starred discretizations, the function $g_1^*$ is as smooth as $f$, so the second expectation is of the same sort we're trying to bound, but one time step earlier than where we started.  

If we define $\varepsilon_n[f] = \E f(\bld{S}^{c,l}_n[\bld{x}]) - \E f(\bld{S}^{*,l}_n[\bld{x}])$, then (\ref{err1}) reads
\begin{equation}
	\varepsilon_{n+1}[f] = \E \left\{ g^c_1 \left( \bld{S}^{c,l}_n[\bld{x}] \right) - g^*_1\left( \bld{S}^{c,l}_n[\bld{x}] \right) \right\} + \varepsilon_{n}[g^*_1].
\end{equation}
In the same way, we may go on to derive 
\begin{equation}
	\varepsilon_n[g_1^*] = \E \left\{ h^c_1 \left( \bld{S}^{c,l}_n[\bld{x}] \right) - h^*_1\left( \bld{S}^{c,l}_n[\bld{x}] \right) \right\} + \varepsilon_{n-1}[h_1^*]
\end{equation}
for appropriate definitions of $h^c_n$ and $h^*_n$.  Iterating this process, we find
\begin{equation} \label{toterr}
	\varepsilon_n [f] = \sum_{k=1}^{n-1} \E \left\{  h^{k,c}_1 \left( \bld{S}^{c,l}_k[\bld{x}] \right) - h^{k,*}_1\left( \bld{S}^{c,l}_k[\bld{x}] \right) \right\}
\end{equation}
for some sequences of functions $\{h^{k,c}_1\}$ and $\{h^{*,c}_1\}$, each of which is as smooth as $f$ and represents the error in the given function after a single time step (we've used the fact that $\bld{S}^{c,l}_0 = \bld{S}^{*,l}_0$).    This completes the first step of expressing the total error in terms of local truncation errors.  

It just remains to show that each of these errors is $O(\Delta t^2)$.  We do this by Taylor expansion of $f$.  Suppose $f$ has four continuous derivatives, so that we can write out its fourth order Taylor series:
\begin{equation}
\begin{split}
	f\left( \bld{S}_1^{c,l} [\bld{x}] \right) &= f(\bld{x}) + \nabla f(\bld{x}) \cdot \bld{D}^c(\bld{x}, \bld{x}, t_0, \Delta t, \delta W_0, \delta W_{\frac{1}{2}}) \\ 
	&+ \frac{1}{2!} \nabla^2 f(\bld{x}) : \bld{D}^c \bld{D}^c + \frac{1}{3!}\nabla^3 f(\bld{x}) (\bld{D}^c)^3 + \frac{1}{4!}\nabla^4 f(\bld{\xi}) (\bld{D}^c)^4
\end{split}
\end{equation}
for some $\bld{\xi}$ on the line between $\bld{x}$ and $\bld{S}^{c,l}_1$, and $\bld{D}^c$ has the same arguments in all instances.  A similar expression holds for $f(\bld{S}_1^{*,l}[\bld{x}])$.  Subtracting the two expressions and taking expectations, we have
\begin{equation} \label{localtrunc}
\begin{split}
	\E f\left( \bld{S}_1^{c,l} [\bld{x}] \right) - \E f\left( \bld{S}_1^{*,l} [\bld{x}] \right) &= \frac{1}{2!}\nabla^2 f(\bld{x}) : \E \left\{ \bld{D}^c \bld{D}^c - \bld{D}^f \bld{D}^f \right\} \\ 
	&+ \frac{1}{3!}\nabla^3 f(\bld{x}) \E \left\{ (\bld{D}^c)^3 - (\bld{D}^f)^3 \right\}\\ 
	&+ \frac{1}{4!}\E \left\{ \nabla^4 f(\bld{\xi}_1) (\bld{D}^c)^4 \right\} - \E \left\{ \nabla^4 f(\bld{\xi}_2)  (\bld{D}^f)^4 \right\}.
\end{split}
\end{equation}
Careful but straightforward examination of all the expectations in (\ref{localtrunc}) reveals that the lowest order terms that don't vanish in expectation are all $O(\Delta t^2)$.  This is true for any function as smooth as $f$, and so is true of each $h^{k,c}_1$ and $h^{k,*}_1$ in (\ref{toterr}).  The number of terms in the sum in (\ref{toterr}) is $O(\Delta t^{-1})$, so we have the desired result. 


\section{Proof of Theorem \ref{varscale}}
Using Lemmas \ref{fineM2} and \ref{coarseM2}, we may write
\begin{equation} \label{difference2}
\begin{split}
	S^f_{i,n} - S^c_{i,n} &= \left(S^f_{i,n-1} - S^c_{i,n-1}\right) + \left[ a_i(\bld{S}^f_{n-1}) - a_i(\bld{S}^c_{n-1}) \right] \Delta t \\
	&+ \sum_{j=1}^D \left[ b_{ij}(\bld{S}^f_{n-1}) - b_{ij}(\bld{S}^c_{n-1}) \right] \Delta W_{j,n-1} \\
	&+ \sum_{j,k=1}^D \left[h_{ijk}(\bld{S}^f_{n-1}) - h_{ijk}(\bld{S}^c_{n-1}) \right] L_{jk,n-1} \\
	&+ M_{i,n-1} + N_{i,n-1}.
\end{split}
\end{equation}
where $M_{i,n} = M^f_{i,n} + M^c_{i,n}$ and similarly for $N_{i,n}$, and
\begin{equation}
	L_{jk,n} = \Delta W_{j,n} \Delta W_{k,n} - \Omega_{jk} \Delta t.
\end{equation}

If we add up (\ref{difference2}) all the way back to the initial time and use $\bld{S}^f_0 = \bld{S}^c_0$, we have
\begin{equation}
\begin{split}
	S^f_{i,n} - S^c_{i,n} &= \sum_{m=0}^{n-1} \left[ a_i(\bld{S}^f_m) - a_i(\bld{S}^c_m) \right] \Delta t \\
	&+ \sum_{m=0}^{n-1}\sum_{j=1}^D \left[ b_{ij}(\bld{S}^f_m) - b_{ij}(\bld{S}^c_m) \right] \Delta W_{j,m} \\
	&+ \sum_{m=0}^{n-1} \sum_{j,k=1}^D \left[h_{ijk}(\bld{S}^f_m) - h_{ijk}(\bld{S}^c_m) \right] L_{jk,m} \\
	&+ \sum_{m=0}^{n-1} (M_{i,m} + N_{i,m}).
\end{split}
\end{equation}
This is conceptually identical to the second equation in the proof of theorem 4.10 (Appendix 4) in \cite{giles2012antithetic}, and may be treated with exactly the same methods found therein.  

In particular, defining 
\begin{equation}
	R_n = \E \left[ \max_{m\leq n} \norm{\bld{S}^f_m - \bld{S}^c_m}^2 \right],
\end{equation}
one can establish the recursive relation
\begin{equation}
	R_n \leq C \left( \Delta t^2 + \Delta t \sum_{m=0}^{n-1} R_m \right)
\end{equation}
for some $C>0$.  A discrete version of the Gr\"{o}nwall inequality implies
\begin{equation}
	R_n \leq C \left( \Delta t^2 + \sum_{m=0}^{n-1} \Delta t^3 \exp\left\{ \sum_{k=m}^{n-1} \Delta t \right\} \right) \leq C \left( \Delta t^2 + n \Delta t^3 \exp(n\Delta t) \right).  
\end{equation}
Letting $n \rightarrow N$, we have
\begin{equation}
	R_N \leq C (1 + T \exp T ) \Delta t^2,
\end{equation}
which immediately implies the desired result. 


\section{Proof of Lemma \ref{arbM}}
The first step is to reestablish Lemma \ref{fineM2} in the case when the fine steps are allowed to have different step sizes.  The desired result will follow by induction, in that we will treat the first $r$ time-steps as a single step, and the $(r+1)^\textrm{st}$ as the second step.  Let $\delta t_1$ be the time step for the first fine step, and $\delta t_2$ the time step for the second, with $\delta W_{j,n}$ and $\delta W_{j,n+\frac{1}{2}}$ the corresponding Brownian increments with variances $\delta t_1$ and $\delta t_2$, respectively.  That is, 
\begin{equation}
	\bld{S}^f_{n+\frac{1}{2}} = \bld{S}^f_n + \bld{D}^f(\bld{S}^f_n, t_n, \delta t_1, \delta W_{j,n}),
\end{equation}
\begin{equation}
	\bld{S}^f_{n+1} = \bld{S}^f_{n+\frac{1}{2}} + \bld{D}^f(\bld{S}^f_{n+\frac{1}{2}}, t_n+\delta t_1, \delta t_2, \delta W_{j,n+\frac{1}{2}}),
\end{equation}
where we've omitted the $l$ superscripts.  

Through diligent algebra, we can show that 
\begin{equation}
\begin{split}
	S_{i,n+1}^f &= S_{i,n}^f + D^f_i (\bld{S}^f_n,t_n,\delta t_1 + \delta t_2, \delta W_{j,n} + \delta W_{j,n+\frac{1}{2}}) \\
	&-\sum_{j,k=1}^D h_{ijk,n} \left( \delta W_{j,n} \delta W_{k,n+\frac{1}{2}} - \delta W_{k,n} \delta W_{j,n+\frac{1}{2}} \right) \\
	&+ R_{i,n} + M_{i,n}^{(1)} + M_{i,n}^{(2)}
\end{split}
\end{equation}
where
\begin{equation}
\begin{split}
	R_{i,n} &= \left(a_{i,n+\frac{1}{2}} - a_{i,n} \right) \delta t_2 \\
	M_{i,n}^{(1)} &= \sum_{j=1}^D \left( b_{ij,n+\frac{1}{2}} - b_{ij,n} - 2 \sum_{k=1}^D h_{ijk,n} \delta W_{k,n} \right) \delta W_{j,n+\frac{1}{2}} \\
	M_{i,n}^{(2)} &= \sum_{j,k=1}^D \left( h_{ijk,n+\frac{1}{2}} - h_{ijk,n} \right) \left( \delta W_{j,n+\frac{1}{2}} \delta W_{k,n+\frac{1}{2}} - \Omega_{jk} \delta t_2 \right).
\end{split}
\end{equation}
From here, the argument bounding the remainder terms proceeds exactly as in Lemma 4.7 in \cite{giles2012antithetic}.  In particular, $M_{i,n}^{(1)}$ and $M_{i,n}^{(2)}$ have vanishing expectation, and may be shown to scale like $\Delta t^{3/2}$ by Taylor expanding $b_{ij,n+\frac{1}{2}}$ and $h_{ijk,n+\frac{1}{2}}$ about $t_n$.  Similarly, $a_{i,n+\frac{1}{2}}$ is Taylor expanded to separate $R_{i,n}$ into two terms, one of which satisfies the appropriate scaling for $M_{i,n}$ while the other satisfies the scaling for $N_{i,n}$.  We refer the interested reader to \cite{giles2012antithetic} for a thorough treatment.  

We now proceed by induction: Suppose that for some $r < M$, we've shown that
\begin{equation} \label{inductee}
\begin{split}
	S^{f,l}_{i,n+\frac{r}{M}} &= S^{f,l}_{i,n} + D_i^f \left( \bld{S}^{f,l}_n, t_n, r\delta t, \sum_{q=0}^{r-1} \delta W_{j,n+\frac{q}{M}} \right) \\
	&- \sum_{j,k=1}^D h_{ijk}(\bld{S}^{f,l}_n) \left(\mathcal{A}_{jk,n}^{(r)} - \mathcal{A}_{kj,n}^{(r)}\right) \\
	&+ M^f_{i,n} + N^f_{i,n} \\
\end{split}
\end{equation}
where
\begin{equation}
	\mathcal{A}_{jk,n}^{(r)} = \sum_{m=1}^{r-1} \left( \delta W_{k,n+\frac{m}{M}} \sum_{q=0}^{m-1} \delta W_{j,n+\frac{q}{M}} \right).
\end{equation}
and $M^f_{i,n}$ and $N^f_{i,n}$ have the scalings stated in the lemma.  Note that the base case $r=1$ is trivial, and that $r=2$ is given by Lemma \ref{fineM2}.  Then, applying our modified version of Lemma \ref{fineM2} to (\ref{inductee}) and
\begin{equation}
	S^{f,l}_{i,n+\frac{r+1}{M}} = S^{f,l}_{i,n+\frac{r}{M}} + D_i^f \left( \bld{S}^{f,l}_{n+\frac{r}{M}},t_n + r\delta t, \delta t, \delta W_{j,n+\frac{r}{M}} \right),
\end{equation}
we have
\begin{equation}
\begin{split}
	S^{f,l}_{i,n+\frac{r+1}{M}} &= S^{f,l}_{i,n} + D_i^f \left( \bld{S}^{f,l}_n, t_n, r\delta t, \sum_{q=0}^{r} \delta W_{j,n+\frac{q}{M}} \right) \\
	&- \sum_{j,k=1}^D h_{ijk}(\bld{S}^{f,l}_n) \left[ \delta W_{k,n+\frac{r}{M}} \left( \sum_{q=0}^{r-1} \delta W_{j,n+\frac{q}{M}} \right) - \delta W_{j,n+\frac{r}{M}} \left( \sum_{q=0}^{r-1} \delta W_{k,n+\frac{q}{M}} \right) \right] \\
	&- \sum_{j,k=1}^D h_{ijk}(\bld{S}^{f,l}_n) \left(\mathcal{A}_{jk,n}^{(r)} - \mathcal{A}_{kj,n}^{(r)}\right) \\
	&+ M^f_{i,n} + N^f_{i,n} \\
\end{split}
\end{equation}
where the new remainder terms have been absorbed into the $M^f_{i,n}$ and $N^f_{i,n}$.  The second and third lines can be combined to obtain
\begin{equation}
\begin{split}
	S^{f,l}_{i,n+\frac{r+1}{M}} &= S^{f,l}_{i,n} + D_i^f \left( \bld{S}^{f,l}_n, t_n, r\delta t, \sum_{q=0}^{r} \delta W_{j,n+\frac{q}{M}} \right) \\
	&- \sum_{j,k=1}^D h_{ijk}(\bld{S}^{f,l}_n) \left(\mathcal{A}_{jk,n}^{(r+1)} - \mathcal{A}_{kj,n}^{(r+1)}\right) \\
	&+ M^f_{i,n} + N^f_{i,n} \\
\end{split}
\end{equation}
Letting the induction carry to $M$, we have the desired result, for $\mathcal{A}_{jk,n}^{(M)} = \mathcal{A}_{jk,n}$ by definition.


\bibliographystyle{plain}
\bibliography{mlmcpaper}
\end{document}